\pdfoutput=1

\documentclass{amsart}

\usepackage[mathscr]{eucal}
\usepackage{amssymb}
\usepackage[usenames,dvipsnames]{xcolor} 
\usepackage[normalem]{ulem}
\usepackage{amsthm}
\usepackage{bbold}
\usepackage{comment}
\usepackage{enumitem}
\usepackage{amsmath}
\usepackage{wasysym}
\usepackage{tikz-cd}
\usepackage{calc}

\usepackage{etoolbox} 

\usepackage{stmaryrd}

\usepackage[unicode]{hyperref} 

\definecolor{dark-red}{rgb}{0.5,0.15,0.15}
\definecolor{dark-blue}{rgb}{0.15,0.15,0.6}
\definecolor{dark-green}{rgb}{0.15,0.6,0.15}
\hypersetup{
    colorlinks, linkcolor=OliveGreen,
    citecolor=OliveGreen, urlcolor=OliveGreen
}

\usepackage[nameinlink,capitalise,noabbrev]{cleveref}
\usepackage[hyphenbreaks]{breakurl}

\numberwithin{equation}{section}
\usepackage[all]{xy}
\xyoption{line}
\usepackage{graphicx}
\usepackage{mathtools}
\SelectTips{cm}{10}

\newtheorem{Thm}[subsection]{Theorem}
\newtheorem*{Thm*}{Theorem}
\newtheorem{Prop}[subsection]{Proposition}
\newtheorem{Lem}[subsection]{Lemma}
\newtheorem{Cor}[subsection]{Corollary}

\newtheorem{thmx}{Theorem}

\theoremstyle{remark}
\newtheorem{Def}[subsection]{Definition}

\newtheorem{Conv}[subsection]{Convention}

\newtheorem{Rem}[subsection]{Remark}

\let\realequation\equation
\def\equation{\setcounter{equation}{\arabic{subsection}}%
   \refstepcounter{subsection}%
   \realequation}

\usepackage{tikz}
\usepackage{tikz-cd}
\usetikzlibrary{trees}
\usetikzlibrary[shapes]
\usetikzlibrary[arrows]
\usetikzlibrary{patterns}
\usetikzlibrary{fadings}
\usetikzlibrary{backgrounds}
\usetikzlibrary{decorations.pathreplacing}
\usetikzlibrary{decorations.pathmorphing}
\usetikzlibrary{positioning}
\usetikzlibrary{shapes.geometric}

\tikzstyle{category} = [rectangle, rounded corners, minimum width=3cm, minimum height=1cm, text centered, text width=2.5cm, draw=black]
\tikzstyle{area} = [rectangle, minimum width=4cm, minimum height=1cm, text centered, text width=3.8cm, draw=black]

\tikzstyle{arrow} = [thick,->,>=stealth]
\tikzstyle{arrow2} = [thick,->,>=stealth,dotted]

\newcommand{\nc}{\newcommand}
\nc{\dmo}{\DeclareMathOperator}

\usepackage{todonotes}

\nc{\overbar}[1]{\mkern 1.5mu\overline{\mkern-1.5mu#1\mkern-1.5mu}\mkern 1.5mu}

\nc{\kappaaux}{g}
\nc{\kappam}{{\kappaaux({\frak m})}}
\nc{\kappaP}{{\kappaaux(\cat P)}}
\nc{\kappaQ}{{\kappaaux(\cat Q)}}
\nc{\kappaSP}{{\kappaaux_{\cat S}(\cat P)}}
\nc{\kappaTP}{{\kappaaux_{\cat T}(\cat P)}}
\nc{\kappaSQ}{{\kappaaux_{\cat S}(\cat Q)}}
\nc{\kappaTQ}{{\kappaaux_{\cat T}(\cat Q)}}
\nc{\kappaphiB}{{\kappaaux(\varphi(\cat B))}}
\nc{\kappaphiQ}{{\kappaaux(\varphi(\cat Q))}}

\dmo{\Sub}{Sub}
\nc{\SpEn}{\cat S_{E(n)}}
\nc{\SpEnf}{\cat S_n}
\nc{\Loco}[1]{\Loc_{\otimes}\hspace{-0.3ex}\langle #1 \rangle}
\nc{\bbullet}{{\scriptscriptstyle\hspace{-1pt}\bullet}}
\nc{\bullett}{{\scriptscriptstyle\bullet}\hspace{-1pt}}
\nc{\LF}{L\hspace{-0.2ex}F}
\nc{\SpG}{\Sp_G}
\nc{\SpGn}{\Sp_{G,n}}
\nc{\EG}{\bbE_G}
\nc{\EH}{\bbE_H}
\nc{\DEG}{\Der(\EG)}
\nc{\DEH}{\Der(\EH)}
\nc{\DE}{\Der(\bbE)}
\nc{\Prst}{{\cat P}\mathrm{r^{st}}}
\nc{\Mack}[2]{\mathrm{Mack}_{#1}(#2)}
\nc{\SC}{S\cat C}
\dmo{\fin}{{fin}}
\dmo{\DM}{DM}
\dmo{\fp}{fp}
\nc{\DMQ}{\DM_Q}
\dmo{\DerKal}{DMack}
\dmo{\Der}{D}
\dmo{\DMot}{DMot}
\dmo{\rmH}{H}
\dmo{\piu}{\underline{\pi}}
\dmo{\Sphere}{\mathbb{S}}
\nc{\HA}{{\rmH \hspace{-0.2em}\bbA}}
\nc{\HZ}{{\rmH \hspace{-0.2em}\bbZ}}
\nc{\HZbar}{{\rmH \hspace{-0.2em}\underline{\bbZ}}}
\nc{\Fp}{{\bbF_{\hspace{-0.1em}p}}}
\nc{\HFp}{{\rmH \hspace{-0.15em}\bbF_{\hspace{-0.1em}p}}}
\nc{\DHZG}{\Der(\HZ_G)}
\nc{\DHZH}{\Der(\HZ_H)}
\nc{\DHZK}{\Der(\HZ_K)}
\nc{\DHZGN}{\Der(\HZ_{G/N})}
\nc{\DHZGG}{\Der(\HZ_{G/G})}
\nc{\DHZCp}{\Der(\HZ_{C_p})}
\nc{\DHZGprime}{\Der(\HZ_{G'})}
\nc{\DHZ}{\Der(\HZ)}
\nc{\frakp}{\mathfrak{p}}
\nc{\frakq}{\mathfrak{q}}
\nc{\Z}{\mathbb{Z}}
\nc{\F}{\mathbb{F}}
\nc{\SSG}{\text{sSet}_*^G}
\nc{\sSet}{\text{sSet}}

\dmo{\csupp}{supp_{coh}}
\dmo{\Con}{Conj}
\dmo{\Id}{Id}
\dmo{\Loc}{Loc}
\dmo{\rmK}{\textrm{\rm K}}
\dmo{\Spc}{Spc}
\dmo{\thick}{thick}
\dmo{\Thick}{Thick}
\nc{\Thickt}[1]{\Thick_\otimes\langle #1 \rangle}
\dmo{\cone}{cone}
\dmo{\End}{End}
\dmo{\Mor}{Mor}
\dmo{\Hom}{Hom}
\dmo{\id}{id}
\dmo{\incl}{incl}
\dmo{\Img}{Im}
\dmo{\im}{im}
\dmo{\Ker}{Ker}
\dmo{\ind}{ind}
\dmo{\CoInd}{coind}
\dmo{\res}{res}
\dmo{\infl}{infl}
\dmo{\triv}{triv}
\dmo{\Tel}{Tel} 
\dmo{\grMod}{grMod}%
\dmo{\Mod}{Mod}%
\dmo{\opname}{op}
\dmo{\SH}{SH}
\dmo{\smallb}{b}
\dmo{\Spec}{Spec}
\dmo{\supp}{supp}
\dmo{\Supp}{Supp}
\dmo{\cosupp}{cosupp}
\dmo{\Cosupp}{Cosupp}
\nc{\SHc}{{\SH^c}}
\nc{\SHp}{{\SH_{(p)}}}
\nc{\SHcp}{{\SH^c_{(p)}}}
\nc{\SHG}{\SH(G)}
\nc{\SHGp}{\SH(G)_{(p)}}
\nc{\SHGc}{\SHG^c}
\nc{\SHGcp}{\SHG^c_{(p)}}
\nc{\quadtext}[1]{\quad\textrm{#1}\quad}
\nc{\qquadtext}[1]{\qquad\textrm{#1}\qquad}
\nc{\adj}{\dashv}
\nc{\adjto}{\rightleftarrows}
\nc{\bbL}{\mathbb{L}}
\nc{\bbA}{\mathbb{A}}
\nc{\bbE}{\mathbb{E}}
\nc{\bbN}{\mathbb{N}}
\nc{\bbQ}{\mathbb{Q}}
\nc{\bbZ}{\mathbb{Z}}
\nc{\bbF}{\mathbb{F}}
\nc{\cat}[1]{\mathscr{#1}}
\nc{\ie}{{\sl i.e.}, }
\nc{\into}{\mathop{\rightarrowtail}}
\nc{\inv}{^{-1}}
\nc{\isoto}{\mathop{\overset{\sim}\to}}
\nc{\isotoo}{\mathop{\overset{\sim}\too}}
\nc{\onto}{\mathop{\twoheadrightarrow}}
\nc{\too}{\mathop{\longrightarrow}\limits}
\nc{\mapstoo}{\longmapsto}
\nc{\adh}[1]{\overline{#1}}
\nc{\adhpt}[1]{\adh{\{#1\}}}
\nc{\aka}{{a.\,k.\,a.}\ }
\nc{\calF}{\mathcal{F}}
\nc{\eg}{{\sl e.\,g.}}
\nc{\Homcat}[1]{\Hom_{\cat #1}}
\nc{\hook}{\hookrightarrow}
\nc{\ideal}[1]{\langle #1\rangle}
\nc{\ihom}{{\underline{\hom}}}
\nc{\iHom}{\mathcal{H}\mathrm{om}}
\nc{\Mid}{\,\big|\,}
\nc{\MMod}{\,\text{-}\Mod}%
\nc{\op}{^{\opname}}
\nc{\oto}[1]{\overset{#1}\to}
\nc{\otoo}[1]{\overset{#1}{\,\too\,}}
\nc{\sminus}{\!\smallsetminus\!}
\nc{\poplus}[1]{^{\oplus #1}}%
\nc{\potimes}[1]{^{\otimes #1}}
\nc{\sbull}{{\scriptscriptstyle\bullet}}
\nc{\SET}[2]{\big\{\,#1\Mid#2\,\big\}}
\nc{\SpcK}{\Spc(\cat K)}
\nc{\then}{\Rightarrow}
\nc{\unit}{\mathbb{1}}
\nc{\xra}{\xrightarrow}
\nc{\phigeom}[1]{\widetilde{\Phi}^{#1}}
\nc{\phigeomb}[1]{\Phi^{#1}}
\dmo{\Oname}{O}
\dmo{\proper}{proper}
\dmo{\lenormal}{\unlhd}
\dmo{\lnormal}{\lhd}
\nc{\normal}{\trianglelefteq}
\nc{\Op}{\Oname^p}
\nc{\Oq}{\Oname^q}
\dmo{\Sp}{Sp}
\dmo{\Ho}{Ho}
\dmo{\Fin}{Fin}
\dmo{\add}{add}
\dmo{\Fun}{Fun}
\dmo{\Ext}{Ext}
\dmo{\CAlg}{CAlg}
\dmo{\CMon}{CMon}
\dmo{\CC}{\cat C}
\dmo{\DD}{\cat D}
\dmo{\OO}{\mathcal{O}}
\dmo{\Map}{Map}
\dmo{\Span}{Span}
\dmo{\N}{N}
\dmo{\Cat}{Cat}
\dmo{\colim}{colim}
\dmo{\hocolim}{hocolim}
\dmo{\Ch}{Ch}
\dmo{\A}{\mathbb{A}^{eff}}
\nc{\AGeff}{\mathbb{A}_G^{\mathrm{eff}}}
\nc{\BGeff}{\mathcal{B}_G^{\mathrm{eff}}}
\nc{\BG}{{\mathcal{B}_G}}
\nc{\NBGeff}{{\N}{\BGeff}}
\dmo{\Ab}{Ab}
\dmo{\Set}{Set}
\dmo{\ev}{ev}
\dmo{\Spcl}{Spcl}
\nc{\Funadd}{\Fun_{\add}}
\dmo{\proj}{proj}
\dmo{\cof}{cof}

\dmo{\Coideal}{Coideal}
\dmo{\gen}{gen}
\dmo{\StMod}{StMod}

\dmo{\projmod}{Lat}
\dmo{\lat}{lat}
\dmo{\Lat}{Lat}
\dmo{\rep}{rep}
\dmo{\Rep}{Rep}
\dmo{\Perf}{Perf}
\dmo{\stmod}{stmod}
\dmo{\Ind}{Ind}
\nc{\borel}[1]{\underline{#1}}
\dmo{\coind}{coind}
\dmo{\rank}{rank}
\nc{\tH}{\hat{H}}
\dmo{\Nm}{Nm}
\dmo{\Proj}{Proj}
\dmo{\Inj}{Inj}
\dmo{\dual}{dual}
\dmo{\fg}{fg}
\nc{\cdvr}[2]{{#1}_{#2}^{\wedge}}
\nc{\cA}{\mathcal{A}}
\dmo{\orbit}{Or}

\nc{\mT}{\kern-0.5em\mod\kern-0.1em\text{-}\cat{T}^c}
\nc{\MT}{\Mod\kern-0.1em\text{-}\cat{T}}
\newcounter{enum-resume-hack}

\dmo{\bP}{\mathbb{P}}
\dmo{\bT}{\mathbb{T}}
\nc{\LOCO}{\mathcal{L}\mathrm{oc}_{\otimes}}

\dmo{\Ob}{Ob}

\begin{document}


\title[Integral stratification]{Stratifying integral representations \\ via equivariant homotopy theory}

\author{Tobias Barthel}

\date{\today}

\makeatletter
\patchcmd{\@setaddresses}{\indent}{\noindent}{}{}
\patchcmd{\@setaddresses}{\indent}{\noindent}{}{}
\patchcmd{\@setaddresses}{\indent}{\noindent}{}{}
\patchcmd{\@setaddresses}{\indent}{\noindent}{}{}
\makeatother

\address{Tobias Barthel, Max Planck Institute for Mathematics, Vivatsgasse 7, 53111 Bonn, Germany}
\email{tbarthel@mpim-bonn.mpg.de}
\urladdr{https://sites.google.com/view/tobiasbarthel/home}

\begin{abstract}
We prove that the derived category of $R$-linear representations of a finite group $G$ is stratified for any regular commutative ring $R$. As an application, we obtain a classification of localizing tensor ideals of ordinary $R$-linear $G$-representations whose underlying $R$-module is projective.
\vspace{-3em}
\end{abstract}


\maketitle

\tableofcontents
\vspace{-3em}

\section{Introduction}

Let $G$ be a finite group and $R$ a commutative ring. The subject of this paper is the classification of $R$-linear $G$-representations. Since such a classification up to isomorphism is inaccessible in general, we instead turn to a coarser classification problem. This takes place in a derived context, more specifically the derived category of $R$-linear $G$-representations $\Rep(G,R)$ from \cite{barthel_integral1}. While constructed homotopically, it also encompasses the algebraic category of representations.

Originating in the ideas of Ravenel, Hopkins, Neeman, and Thomason, tensor-triangular geometry as pioneered by Balmer provides a context in which the local and global structure of tensor-triangulated categories is studied through the classification of their tensor ideals. One approach to such a classification for $\Rep(G,R)$ is based on cohomological support, which assigns to any representation a subset of the homogeneous Zariski spectrum of $H^*(G;R)$. Moreover, it descends to a map 
\[
\begin{Bmatrix}
\text{Localizing tensor ideals of } \Rep(G,R)
\end{Bmatrix} 
\xymatrix@C=2pc{ \ar[r]^-{} &}
\begin{Bmatrix}
\text{Subsets of } \Spec^h(H^*(G;R))
\end{Bmatrix}.
\]
We say that $\Rep(G,R)$ is \emph{stratified} if this map is a bijection. In other words, stratification amounts to a generic classification of representations in terms of cohomological support. The main theorem of this paper is:  

\begin{thmx}[\cref{thm:repstratification}]\label{thma}
The category $\Rep(G,R)$ is stratified for any finite group $G$ and any regular commutative ring $R$.
\end{thmx}

This result generalizes the main theorem of Benson, Iyengar, and Krause \cite{BensonIyengarKrause11a} for $R$ a field, as well as \cite{barthel_integral1} for $R$ a ring of integers in a number field, without depending on these earlier results. The proof uses ideas from our previous paper \cite{barthel_integral1}, but amplifies them through the use of tools from equivariant homotopy theory. It can be seen as a blueprint of similar classification results in \emph{spectral representation theory}: that is, for derived categories of representations with coefficients in ring spectra.

\subsection{Context and consequences of the main theorem}

This paper takes place in an appropriate framework of derived (or spectral) representation theory. The key category $\Rep(G,R)$ and its accompanying stable module category $\StMod(G,R)$ were introduced in \cite{barthel_integral1} as a suitable generalization of their more familiar algebraic cousins $K(\Inj(kG))$ and $\StMod(G,k)$ from fields $k$ to general ring (spectra) coefficients. A previous candidate was constructed by Benson, Iyengar, and Krause \cite{BensonIyengarKrause13}, but was there found to be unsuitable for a cohomological stratification. 

While homotopical in nature, and therefore perhaps less familiar to an algebraic audience, we offer a few remarks on why $\Rep(G,R)$ exhibits a number of useful features in comparison to its more algebraic manifestations (reviewed in \cref{rem:incarnations}):
    \begin{itemize}
        \item The construction naturally extends to more general situations, such as to compact Lie groups instead of finite groups $G$ or with coefficients in ring spectra $R$ as opposed to rings. 
        \item From an abstract point of view, it serves as a gateway between representation theory and stable homotopy theory, and thus allows us to bring homotopical techniques (such as spectral Galois theory) to bear on algebraic problems. 
        \item It makes precise the intuitive idea that derived representations of a group $G$ in a category $\cat C$ should be given by `continuous' functors from the classifying space $BG$ to $\cat C$. Informally speaking, such a functor sends the point of $BG$ to an object of $\cat C$, the elements of $G$ to automorphisms of $G$, and all of this coherently up to higher homotopies. 
    \end{itemize}
Moreover, through its connection to equivariant homotopy theory, $\Rep(G,R)$ becomes particularly amenable to the general theory of stratification developed in \cite{bhs1}, following earlier work of Hovey--Palmieri--Strickland \cite{HoveyPalmieriStrickland97}, Benson--Iyengar--Krause \cite{BensonIyengarKrause11b}, and Stevenson~\cite{Stevenson13}. This general theory also allows us to deduce several important structural features about $\Rep(G,R)$ from the main theorem. In particular, we obtain the next result:

\begin{thmx}[\cref{thm:repconsequences}]
Let $G$ be a finite group and let $R$ be a regular commutative ring. The telescope conjecture holds in $\Rep(G,R)$ and cohomological support satisfies the tensor product formula.
\end{thmx}

In addition, \cref{thma} also has applications to ordinary (as opposed to derived) representations of a finite group. Consider the category $\Lat(G,R)$ of $R$-linear $G$-representations whose underlying $R$-module is projective. Combined with the techniques of \cite{barthel_integral1}, our main theorem leads to a `generic' classification of objects in $\Lat(G,R)$:

\begin{thmx}[\cref{cor:classicalrepconsequences}]
Let $G$ be a finite group and $R$ a regular commutative ring. Cohomological support induces a bijection 
\[
\begin{Bmatrix}
\text{Non-zero localizing tensor } \\
\text{ideals of } \cat \Lat(G,R)
\end{Bmatrix} 
\xymatrix@C=2pc{ \ar[r]^-{\sim} &}
\begin{Bmatrix}
\text{Subsets of}  \\
\Proj(H^*(G;R))
\end{Bmatrix}.
\]
\end{thmx}

A number of further applications are collected in \cref{ssec:applications}.

\subsection{Strategy of proof}

The starting point for our approach to the classification of localizing ideals in the category of $R$-linear $G$-representations is a suitable version of the stable module category for general coefficients. The subtleties between different choices of this category prevented progress on stratification for general coefficients for about a decade, see \cite{BensonIyengarKrause13}. 

In \cite{barthel_integral1}, we therefore introduced and studied a derived category of representations $\Rep(G,R)$, building on earlier ideas of Mathew and A.~Krause. This category is constructed by means of higher algebra, generalizes well to ring spectra coefficients, while also connecting in the expected way to classical algebraic categories of representations. The key new observation of this paper is that the stratification of  $\Rep(G,R)$ may be conveniently studied using equivariant homotopy theory. 

If $\underline{R}_G$ denotes the Borel-completion of $R$ and $\Mod_{\Sp_G}(\underline{R}_G)$ the associated category of equivariant modules, generation by permutation modules due to Rouquier, Mathew, and Balmer--Gallauer yields a symmetric monoidal equivalence
\[
\xymatrix{\Mod_{\Sp_G}(\underline{R}_G) \ar[r]^-{\sim} & \Rep(G,R).}
\]
This translates the representation-theoretic stratification problem to one in equivariant homotopy theory. 

As already observed in \cite{barthel_integral1}, stratification for $\Mod_{\Sp_G}(\underline{R}_G)$ for $R$ a commutative Noetherian ring and any finite group $G$ reduces to the case of elementary abelian subgrous $E$ in $G$. For regular $R$, an elaboration on the modular lifting technique of \cite{barthel_integral1} together with Lau's work \cite{lau_spcdmstacks} then reduces the problem further to coefficients in $R/p$. Importantly, one has to note that $R/p$ is, in general, no longer regular. Nevertheless, unipotence allows us to translate the stratification problem further to the stratification of the non-equivariant category of module spectra over the cochains $C^*(BE,R/p)$. 

Building on an unpublished idea of Mathew, spectral Galois ascent then allows us to replace the elementary abelian group $E$ by a torus of the same rank. A formality result coupled with a stratification theorem for formal dg-algebras due to Benson, Iyengar, and Krause then finishes the proof. 

\subsection{Comments on the methodology}

The reader might wonder to what extent the techniques used here are essential to establish stratification for algebraic categories such as $\Rep(G,R)$ for commutative rings $R$, in particular pertaining to the tools from equivariant homotopy theory and higher algebra. More geodesic approaches to our results would indeed be possible; for example, the passage from an algebraic manifestation of $\Rep(G,R)$ to modules over certain cochain algebras could be modelled purely algebraically. However, with a view towards future applications to questions about spectral representations of finite groups, i.e., $\Rep(G,R)$ with ring spectra coefficients, such a formulation would be insufficient. 

In fact, establishing a connection between representation theory and stable (equivariant) homotopy theory might be of independent interest. From a bird's-eye view, it highlights the utility of the tt-geometry, but it also presents practical advantages. For example, homotopy theoretic techniques allow us to give a short proof of stratification for elementary abelian $p$-groups that works for coeffcients in any Noetherian algebra over a field $k$ of characteristic $p$. While a purely algebraic proof might be possible using an elaboration of the techniques from \cite[Sec.~4--8]{BensonIyengarKrause11a} for $A=k$, our argument works uniformly in this generality. 

In a different direction, an alternative proof of the main theorem inspired by the following idea has recently been announced by Benson, Iyengar, Krause, and Pevtsova \cite{bikp_fiberwisestratification}: For any residue field $\kappa$ of $R$, we can probe the category of representations via tt-functors
\[
\xymatrix{\kappa\otimes_R - \colon \Rep(G,R) \ar[r] & \Rep(G,\kappa).}
\]
Varying $\kappa$ over the Zariski spectrum of $R$, this reduces stratification for regular coefficients to the field case already established in \cite{BensonIyengarKrause11a}. A similar strategy is used in Lau's paper \cite{lau_spcdmstacks} on the Balmer spectrum of Deligne--Mumford stacks. In comparison to such strategies, we highlight three differences of our perspective: 
\begin{enumerate}
    \item The proof given in this paper proceeds from first principles and thereby reproves the main theorem of \cite{BensonIyengarKrause11a} as opposed to taking it as an input.
    \item Our arguments apply more generally to certain categories of spectral representations, i.e., categories $\Rep(G,R)$ with coefficients in commutative ring spectra $R$ or more general types of groups $G$ such as compact Lie groups, for which residue fields might not be readily available. 
    \item The homotopical enrichment is vital for descent techniques. Another instance of this is forthcoming work with Grodal and Hunt, in which we use $\infty$-categorical tools to study endotrivial modules. 
\end{enumerate}
Of particular interest to us is the case of Lubin--Tate spectra $R = E_n$, which is part of joint work in progress with Castellana, Heard, Naumann, and Pol. 

\subsection{Outline of the document}

Unwinding the strategy, an overview of the proof is given schematically in the following diagram:

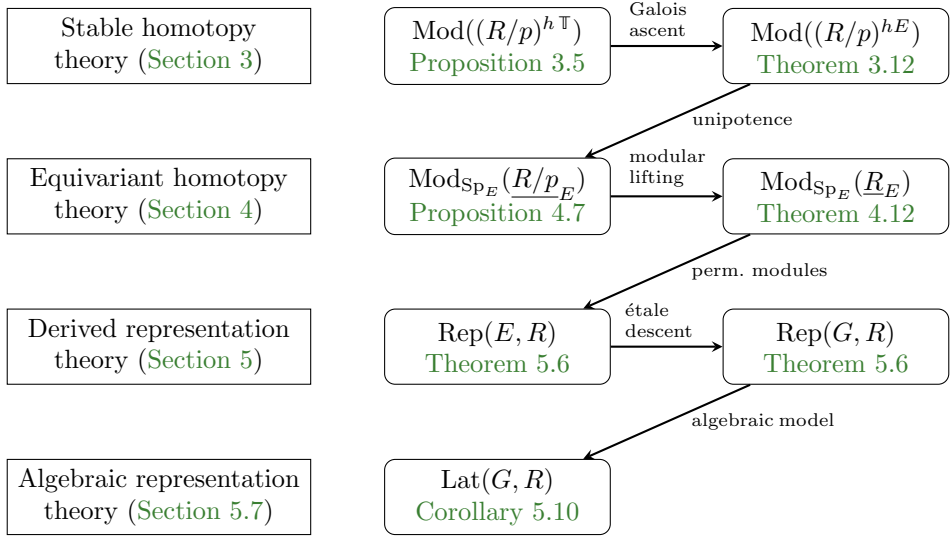
\begin{figure}[h]
\centering
\begin{tikzpicture}[node distance=2cm]
\node (A1) [area] {Stable homotopy \\theory (\cref{sec:stablehomotopy})};
\node (A2) [area, below of=A1] {Equivariant homotopy \\theory (\cref{sec:eqperm})};
\node (A3) [area, below of=A2] {Derived representation \\ theory (\cref{sec:reptheory})};
\node (A4) [area, below of=A3] {Algebraic representation \\ theory (\cref{ssec:applications})};

\node (C1) [category, right of=A1, xshift=2.5cm] {$\Mod((R/p)^{h\bT})$ \\ \cref{prop:torusstratification}};
\node (C2) [category, right of=C1, xshift=2.5cm] {$\Mod((R/p)^{hE})$ \\ \cref{thm:elabstratificationmodp}};

\node (C3) [category, right of=A2, xshift=2.5cm] {$\Mod_{\Sp_{E}}(\underline{R/p}_E)$ \\ \cref{prop:elabeqstratificationcharp}};
\node (C4) [category, right of=C3, xshift=2.5cm] {$\Mod_{\Sp_{E}}(\underline{R}_E)$ \\ \cref{thm:elabeqstratification}};

\node (C5) [category, right of=A3, xshift=2.5cm] {$\Rep(E,R)$ \\ \cref{thm:repstratification}};
\node (C6) [category, right of=C5, xshift=2.5cm] {$\Rep(G,R)$\\ \cref{thm:repstratification}};

\node (C7) [category, right of=A4, xshift=2.5cm] {$\Lat(G,R)$ \\ \cref{cor:classicalrepconsequences}};

\draw [arrow] (C1) -- node[anchor=south] {\scriptsize\parbox{1cm}{Galois ascent}} (C2);
\draw [arrow] (C2) -- node[anchor=west, xshift=0.2cm] {\scriptsize\parbox{1cm}{unipotence}} (C3);
\draw [arrow] (C3) -- node[anchor=south] {\scriptsize\parbox{1cm}{modular lifting}} (C4);
\draw [arrow] (C4) -- node[anchor=west, xshift=0.2cm] {\scriptsize\parbox{1.9cm}{perm.~modules}} (C5);
\draw [arrow] (C5) -- node[anchor=south] {\scriptsize\parbox{1.1cm}{\'etale \\ descent}} (C6);
\draw [arrow] (C6) -- node[anchor=west, xshift=0.2cm] {\scriptsize\parbox{1.9cm}{algebraic model}} (C7);
\end{tikzpicture}
\caption{Leitfaden.}
\end{figure}

The same diagram serves as an outline of the contents of the paper; we refer to the introductions of the individual sections for further details.

\subsection{Conventions and terminology}

In this paper, we usually work with suitable homotopical enrichments of (tensor-)triangulated categories; to fix ideas, our choice is the theory of (symmetric monoidal) stable $\infty$-categories developed by Lurie \cite{HTTLurie,HALurie}. Some parts of the paper could equally well be formulated within the context of triangulated categories only, others (such as Galois ascent) require more structure. Since all categories appearing here admit natural models---in fact, they may be constructed model-independently and internally to $\infty$-categories---this poses no issues. 
In particular, the term \emph{tt-category} usually refers to a model of a tensor-triangulated category, and similarly for tt-functors.

By\emph{commutative ring spectrum} we mean an $\mathbb{E}_{\infty}$-ring spectrum, i.e., a coherently commutative ring spectrum in the symmetric monoidal $\infty$-category of spectra. When we wish to emphasize that the ring in question is an ordinary ring, we will also add the adjective \emph{discrete}, as in \emph{discrete commutative algebra}. 

\subsection*{Acknowledgements}

I would like to thank Scott Balchin, Paul Balmer, Nat\`alia Castellana, Frank Gounelas, Drew Heard, and Henning Krause for discussions related to this work. 

\section{Preliminaries on tt-geometry and stratification}\label{sec:ttgeometry}

This section provides a summary of the relevant material from tensor-triangular geometry and the theory of stratification on which our later arguments are based. 

\subsection{Recollections on tensor-triangular geometry}

We begin with a rapid review of the key players in tensor-triangular geometry for our purposes, mostly to fix notation. For more details and context, we refer the reader to \cite{Balmer05a}.

Let $\cat K$ be an essentially small tensor-triangulated category. Tensor-triangular geometry studies the structure of $\cat K$ by viewing it as a geometric object over an associated spectral topological space, its Balmer spectrum $\Spc(\cat K)$.
The construction of the spectrum is accompanied by a support function which takes values in certain subsets of $\Spc(\cat K)$:
\begin{equation}\label{eq:Balmersupport}
    \xymatrix{\supp\colon \Ob\cat K \ar[r] & \{\text{closed subsets of } \Spc(\cat K)\}.}
\end{equation}
Assuming from now on that $\Spc(\cat K)$ is a Noetherian space, which covers all examples of tt-categories appearing in this paper, the Thomason subsets coincide with the specialization closed ones. Balmer establishes a universal property for $\supp$ and proves that it classifies thick tensor ideals of $\cat K$. Explicitly, if $\cat K$ is rigid, the support function induces a bijection
\[
\supp\colon\begin{Bmatrix}
\text{Thick tensor ideals} \\
\text{of } \cat K
\end{Bmatrix} 
\xymatrix@C=2pc{ \ar[r]^-{\sim} &}
\begin{Bmatrix}
\text{Specialization closed} \\
\text{subsets of } \Spc(\cat K)
\end{Bmatrix}.
\]
As an important approximation to the Balmer spectrum, Balmer \cite{Balmer10b} constructed a comparison map between $\Spc(\cat K)$ and the Zariski spectrum of the ring of endomorphisms of the unit $\unit$ of $\cat K$. In more detail, let $R_{\cat K}^{\bullet} = [\unit,\unit]^{\bullet}$ be the graded commutative endomorphism ring of $\unit$. There is a continuous map
\begin{equation}\label{eq:Balmercomparisonmap}
    \xymatrix{\rho^{\bullet} = \rho_{\cat K}^{\bullet}\colon \Spc(\cat K) \ar[r] & \Spec^h(R_{\cat K}^{\bullet}),}
\end{equation}
where $\Spec^h$ denotes the homogeneous Zariski spectrum. The construction of $\rho^{\bullet}$ is natural in tt-functors; moreover, Balmer proves (\cite[Thm.~7.3]{Balmer10b}) that $\rho_{\cat K}^{\bullet}$ is surjective if $R_{\cat K}^{\bullet}$ is a graded coherent commutative ring.

\subsection{Recollections on stratification}\label{ssec:recstratification}

The main abstract tool used in this paper is the theory of stratification relative to the Balmer spectrum developed in \cite{bhs1}.

\begin{Conv}\label{conv:stratification}
In this subsection, we let $\cat T$ be a rigidly-compactly generated tt-category. In addition, we assume that the spectrum $\Spc(\cat T^c)$ is Noetherian.\footnote{This assumption is mostly for simplicity; the theory extends to weakly Noetherian spectra as well, see \cite{bhs1}.}
\end{Conv}

Under suitable point-set topological assumptions \cite{BalmerFavi11, bhs1} on $\Spc(\cat T^c)$, such as the one stipulated in \cref{conv:stratification}, the support function $\supp$ of \eqref{eq:Balmersupport} can be extended to all of $\cat T$:
\begin{equation}\label{eq:BFsupport}
    \xymatrix{\Supp\colon \Ob\cat T \ar[r] & \bP(\Spc(\cat T^c)) := \{\text{Subsets of } \Spc(\cat T^c)\};}
\end{equation}
we refer to this function as the \emph{Balmer--Favi notion of support}. In particular, it has the property that $\Supp(t) = \supp(t)$ for all $t\in \cat T^c$. Moreover, setting $\Supp(\cat L) = \bigcup_{l \in \cat L}\Supp(l)$, we have $\Supp(\Loco{t}) = \Supp(t)$ for any $t\in \cat T$. Therefore, the Balmer--Favi notion of support induces a function defined on all localizing ideals of $\cat T$, which we will usually denote by $\LOCO(\cat T)$.

\begin{Def}\label{def:stratification}
A compactly generated big tt-category $\cat T$ is said to be \emph{stratified} if the map 
\[
\Supp\colon\begin{Bmatrix}
\text{Localizing tensor ideals of } \cat T
\end{Bmatrix} 
\xymatrix@C=2pc{ \ar[r]^-{\sim} &}
\begin{Bmatrix}
\text{Subsets of } \Spc(\cat T^c)
\end{Bmatrix}
\]
is a bijection.
\end{Def}

\begin{Rem}
More precisely, the stratification condition applies to the collection of all \emph{set-generated} localizing ideals of $\cat T$. In light of \cite[Prop.~3.5]{bhs1}, stratification in this sense implies that all localizing ideals of $\cat T$ are set-generated, which is why we leave the generated-by-a-set property implicit here.
\end{Rem}

Under \cref{conv:stratification}, a tt-category $\cat T$ is stratified if and only if the following \emph{minimality condition} is satisfied for all $\cat P \in \Spc(\cat T^c)$:
\begin{itemize}
    \item[] (Minimality) There are no nonzero proper localizing ideals contained in $\Gamma_{\cat P}\cat T = \SET{t \in \cat T}{\Supp(t) \subseteq \{\cat P\}}$.
\end{itemize}
In joint work with Heard and Sanders, we showed that this is in fact the universal theory of stratification (see \cite[Sec.~7]{bhs1} for details) and establish several desirable properties. Of particular importance are the behaviour under base-change (ascent and descent), discussed in more detail in \cref{ssec:ttascentanddescent}.

Extending work of Hovey, Palmieri, and Strickland \cite{HoveyPalmieriStrickland97}, Benson, Iyengar, and Krause \cite{BensonIyengarKrause08,BensonIyengarKrause11b} have developed a theory that provides more explicit conditions for checking stratification. It depends on an auxiliary Noetherian graded commutative ring $R$ acting on $\cat T$, which is used to construct local cohomology functors $\Gamma_{\frak p}^{R \circlearrowright \cat T}$ for $\frak p \in \Spec^h(R)$. The corresponding \emph{BIK-notion of support} is then defined, for any $t \in \cat T$, by
\[
\Supp^{R \circlearrowright \cat T}(t) = \SET{\frak p \in \Spec^h(R)}{\Gamma_{\frak p}^{R \circlearrowright \cat T}(t) \neq 0}.
\]
The categories $\Gamma_{\frak p}^{R \circlearrowright \cat T}\cat T$ are localizing ideals in $\cat T$. Following \cite{BensonIyengarKrause11b}, we say that $\cat T$ is \emph{BIK-stratified} (with respect to the action of $R$ on $\cat T$) if $\Gamma_{\frak p}^{R \circlearrowright \cat T}\cat T$ is zero or minimal for any $\frak p \in \Spec^h(R)$. With this terminology at hand, we can state the following comparison result \cite[Cor.~7.11]{bhs1}:

\begin{Thm}\label{thm:ttstratificationuniversality}
Suppose $\cat T$ is BIK-stratified by the action of a graded commutative Noetherian ring $R$. Then $\cat T$ is stratified (\cref{def:stratification}) and there exists a canonical homeomorphism $\Supp^{R \circlearrowright \cat T}(\cat T) \xrightarrow{\sim}\Spc(\cat T^c)$. Under this homeomorphism, $\Supp^{R \circlearrowright \cat T}$ coincides with the Balmer--Favi notion of support $\Supp$.

Moreover, if $R = [\unit,\unit]^{\bullet}$ and $\cat T$ is stratified by the canonical action of $R$ on $\cat T$, then Balmer's comparison map \eqref{eq:Balmercomparisonmap},
\[
\xymatrix{\rho^{\bullet}\colon \Spc(\cat T^c) \ar[r]^-{\cong} & \Spec^h(R),}
\]
is a homeomorphism.
\end{Thm}

\subsection{Ascent and descent for stratification}\label{ssec:ttascentanddescent}

Stratification over the Balmer spectrum in the sense of \cref{def:stratification} enjoys a number of good properties, see \cite{bhs1}, including permanence under certain base-change functors. In this section, we review these results and exhibit two new variants, one about ascent along certain spectral Galois extensions, and one regarding descent. We continue to write $\LOCO(\cat T)$ for the collection of (set-generated) localizing tensor ideals in a rigidly-compactly generated tt-category $\cat T$.

\subsubsection*{Galois ascent}

We begin with a novel ascent result for stratification along spectral Galois extensions with Galois group a connected topological group of the homotopy type of a finite CW-complex. This idea originates in unpublished work of Mathew \cite{mathew_torus}; a sharper result removing the connectivity assumption is part of forthcoming work with Castellana, Heard, Naumann, and Pol. 

\begin{Lem}\label{lem:Suppcompatibility}
Let $f^*\colon \cat S \to \cat T$ be a tt-functor with conservative right adjoint $f_*$. If $\cat S$ is stratified, then there are commutative squares
\begin{equation}\label{eq:Suppcompatibility}
    \vcenter{
    \xymatrix{\LOCO(\cat S) \ar[r]^-{\Phi} \ar[d]_-{\Supp} & \LOCO(\cat T) \ar[d]^-{\Supp} \\
    \bP(\Spc(\cat S^c)) \ar[r]_{\varphi^{-1}(-)} & \bP(\Spc(\cat T^c)),}
    }
\end{equation}
where $\varphi\colon \Spc(\cat T^c) \to \Spc(\cat S^c)$ is the map on spectra induced by $f^*$ and we set $\Phi(\cat L) = \Loco{f^*(\cat L)}$ for $\cat L \in \LOCO(\cat S)$.
\end{Lem}
\begin{proof}
Under the assumptions of the lemma, the Avrunin--Scott identities, proven in this context in \cite{bchs1}, establish an equality for any $s \in \cat S$:
\[
\Supp(f^*s) = \varphi^{-1}\Supp(s).
\]
Consider then some localizing ideal $\cat L = \Loco{\SET{s_i}{i\in I}}$ in $\cat S$. We compute:
\begin{align*}
    \Supp(\Phi(\cat L)) & = \Supp(\Loco{\SET{f^*s_i}{i\in I}}) 
    = \bigcup_{i \in I} \Supp(f^*s_i)  
    = \bigcup_{i \in I} \varphi^{-1}\Supp(s_i) \\
    & = \varphi^{-1}\Supp(\Loco{\SET{s_i}{i\in I}}) 
    = \varphi^{-1}\Supp(\cat L),
\end{align*}
as desired.
\end{proof}

We freely use the theory of spectral Galois extensions developed by Rognes \cite{Rognes08}, and refer to his work for their basic properties. In particular, a Galois extension $A \to B$ is called \emph{faithful} if the corresponding induction functor is conservative. 

\begin{Prop}\label{prop:galoisascent}
Suppose $A \to B$ is a faithful $G$-Galois extension of commutative ring spectra with connected Galois group $G$ of the homotopy type of a finite CW-spectrum. If $\Mod(A)$ is stratified, then so is $\Mod(B)$. Moreover, the tt-functor $f^*$ induces a homeomorphism
\[
\xymatrix{\varphi\colon\Spc(\Mod^c(B)) \ar[r]^-{\cong} & \Spc(\Mod^c(A))}.
\]
\end{Prop}
\begin{proof}
The commutative square of \cref{lem:Suppcompatibility} specializes to 
\[
\xymatrix{\LOCO(\Mod(A)) \ar[r]^-{\Phi} \ar[d]_-{\Supp} & \LOCO(\Mod(B)) \ar[d]^-{\Supp} \\
    \bP(\Spc(\Mod^c(A))) \ar[r]_{\varphi^{-1}(-)} & \bP(\Spc(\Mod^c(B))).}
\]
The top and bottom horizontal maps are bijections, using \cite[Thm.~6.5]{mathew_torus} and its compact analogue \cite{bchnp1}, respectively. Since the left vertical map is a bijection by assumption, the right vertical map is a bijection as well. In other words, $\Mod(B)$ is stratified. 
\end{proof}

\subsubsection*{Descent}

The next result assembles and, in the case of Item (2) below, slightly extends some descent results for stratification from \cite{bhs1} and \cite{barthel_integral1}.

\begin{Prop}\label{prop:ttdescent}
Let $f^*\colon \cat S \to \cat T$ be a colimit preserving tt-functor between rigidly-compactly tt-categories with Noetherian spectra, and let $f_*$ be a right adjoint to $f_*$. Write $\varphi\colon \Spc(\cat T^c) \to \Spc(\cat S^c)$ for the map on spectra induced by $f^*$. Suppose that there exists a commutative algebra object $A \in \cat S$ and an equivalence of tt-categories $\cat T \simeq \Mod_{\cat S}(A)$ making the following diagram commute
\[
\xymatrix{\cat S \ar[r]^-{f^*} \ar[dr]_-{-\otimes A} & \cat T \ar[d]^-{\sim} \\
& \Mod_{\cat S}(A).}
\]
In addition, assume one of the following conditions is satisfied:
    \begin{enumerate}
        \item[(1)] 
        $f^*$ is a finite localization away from a Thomason subset $Y \subseteq \Spc(\cat S^c)$;
        \item[(2)] 
        $A \in \cat S$ is compact and $\varphi$ is injective;
        \item[(3)] 
        $A \in \cat S$ is compact and, as an algebra, faithful and separable of finite degree. 
    \end{enumerate}
If $\cat T$ is stratified, then the categories $\Gamma_{\frak p}\cat S$ are minimal for all $\frak p \in \supp(A)$.
\end{Prop}
\begin{proof}
Part (1) is an instance of Zariski descent for stratification, established in \cite[Cor.~5.3 and Rem.~5.4]{bhs1}, while Part (3) is proven in \cite[Thm.~2.14]{barthel_integral1}. Note that, for a finite localization away from $Y$ as in (1), $\supp(A)$ is the complement of $Y$, while in the case of (3), we have $\supp(A) = \Spc(\cat S^c)$.

It remains to verify the claim in the case of (2). Let $A \in \CAlg(\cat S^c)$ and assume that the fiber of $\varphi$ is either empty or singleton. It follows from \cite[Thm.~1.7]{Balmer18} that $\supp(A) = \im(\varphi)$. Let $\frak p \in \supp(A)$ and write $\varphi^{-1}(\frak p) = \{\frak q\}$. By localizing away from $\langle \frak p \rangle$, we may assume that $\frak p$ is maximal in $\cat C$, see for example \cite[Ex.~5.7]{Sanders21pp}. The proofs of Lemmas 2.17, 2.18, and 2.19 of \cite{barthel_integral1} go through without the assumption that $A$ is a separable algebra of finite degree, using instead that $\varphi^{-1}(\frak p)$ is a singleton. We conclude that $\Gamma_{\frak p}\cat S$ is minimal if $\Gamma_{\frak q}\cat T$ is minimal, which holds by assumption. This finishes the proof.
\end{proof}

\section{Cochain algebras with general coefficients}\label{sec:stablehomotopy}

Our first goal is to study stratification for commutative ring spectra of cochains for tori and elementary abelian groups with algebraic coefficients. The results of this section will provide the base case for applying our descent techniques in later sections; in particular, we do not strive for maximum generality here. The two main ingredients are a stratification theorem for formal dg-algebras due to Benson, Iyengar, and Krause, as well as Rognes' theory of spectral Galois extensions to pass from tori to elementary abelian groups.  A good introduction to cochain algebras from the representation theoretic point of view is contained in \cite{bensonkrause_complexes}.

\subsection{Tori}

The starting point for our approach to stratification of the derived category of representations is an especially simple category: the category of module spectra over the cochains of a torus.

\begin{Conv}
Throughout this subsection, $R$ denotes a commutative ring, not necessarily regular but later assumed to be Noetherian. Furthermore, $\bT = (S^1)^{\times r}$ is a torus of rank $r$, which acts trivially on $R$. The resulting homotopy fixed point spectrum $R^{h\bT}$ has the structure of a commutative ring spectrum which can be modelled by the commutative ring spectrum of cochains $C^*(B\bT,R)$. 
\end{Conv}

We begin with a standard formality result about cochains on tori; for the convenience of the reader, we outline a proof. 

\begin{Lem}\label{lem:formality}
For any commutative ring $R$, the underlying $\mathbb{E}_1$-$R$-algebra of $R^{h\bT}$ is formal. In algebraic terms, there is an equivalence 
\[
C^*(B\bT,R) \simeq \pi_{*}R^{h\bT}
\]
of dg-algebras over $R$, where the graded cohomology ring $\pi_{-*}R^{h\bT} \cong H^{*}(\bT;R)$ is equipped with the zero differential. 
\end{Lem}
\begin{proof}
The graded commutative cohomology ring of $R^{h\bT}$ is isomorphic to a graded polynomial ring over $R$,
\begin{equation}\label{eq:cohomologytorus}
\pi_*R^{h\bT} \cong R[x_1,\ldots,x_r],
\end{equation}
where each of the generators $x_i$ has (homological) degree $-2$. Represent these generators by maps $x_i\colon\Sigma^{-2}R \to R^{h\bT}$ of $R$-module spectra. The induced map 
\[
\textstyle\bigoplus_{i=1}^r\Sigma^{-2}R \to R^{h\bT}
\]
extends to a map of $\mathbb{E}_1$-$R$-algebras from the free $\mathbb{E}_1$-$R$-algebra on $\bigoplus_{i=1}^r\Sigma^{-2}R$:
\[
\mathrm{Free}_{\mathbb{E}_1-R}(x_1,\ldots,x_r) \to R^{h\bT}.
\]
By construction, this map induces an isomorphism on homotopy groups and thus exhibits the underlying $\mathbb{E}_1$-$R$-algebra of $R^{h\bT}$ as formal. Finally, Shipley's theorem \cite{shipley_hzalgebras} provides the translation into the language of dg-algebras. 
\end{proof}

\begin{Prop}\label{prop:torusstratification}
Let $R$ be a Noetherian commutative ring and $\bT = (S^{1})^{\times r}$ a torus of rank $r$. The tt-category $\Mod(R^{h\bT})$ is stratified over the graded affine space
\begin{equation}\label{eq:spctorus}
\Spc(\Mod^c(R^{h\bT})) \cong \Spec^h(\pi_*R^{h\bT}) \cong \Spec^h(R[x_1,\ldots,x_r]),
\end{equation}
where $|x_i|=-2$ for all $i$.
\end{Prop}
\begin{proof}
We employ the theory of BIK-stratification \cite{BensonIyengarKrause11b}, brielfy reviewed in \cref{ssec:recstratification} above, to prove this result; in light of \cref{thm:ttstratificationuniversality}, this suffices to show that $\Mod(R^{h\bT})$ is also stratified in the tt-sense. 

Consider the tt-category $\Mod(R^{h\bT})$ equipped with the canonical action by $\pi_*R^{h\bT}$. As an $\mathbb{E}_1$-algebra, $\pi_*R^{h\bT}$ is formal by \cref{lem:formality}; furthermore, it is Noetherian because $R$ is Noetherian. The equivalence of $\mathbb{E}_1$-algebras $R^{h\bT} \simeq \pi_*R^{h\bT}$ induces an equivalence of stable $\infty$-categories
\[
\Mod(R^{h\bT}) \simeq \Mod(\pi_*R^{h\bT}).
\]
Via the transfer principle of \cite[Cor.~8.4]{BensonIyengarKrause11b}, it suffices to show that $\Mod(\pi_*R^{h\bT})$ is stratified by the induced action of $\pi_*R^{h\bT}$. Since $\pi_*R^{h\bT}$ is a commutative dg-algebra which is formal as a dg-algebra, \cite[Thm.~8.1]{BensonIyengarKrause11b} applies to give the desired stratification result. 

It therefore remains to determine the Balmer spectrum of $\Mod(R^{h\bT})$. To this end, observe that the residue fields $\kappa(\frak p)$ of $\pi_*R^{h\bT}$ are non-trivial for all $\frak p \in \Spec^h(\pi_*R^{h\bT})$, hence the BIK-support of $\Mod(R^{h\bT})$ is all of $\Spec^h(\pi_*R^{h\bT})$. The first isomorphism of \eqref{eq:spctorus} thus follows from \cref{thm:ttstratificationuniversality}, while the second is a consequence of \eqref{eq:cohomologytorus}.
\end{proof}

\begin{Rem}
We emphasize that \cref{prop:torusstratification} works for coefficients in an arbitrary Noetherian commutative ring $R$. It is for this reason that the abstract results of \cite{DellAmbrogioStanley16}, for example, are not readily applicable.
\end{Rem}


\subsection{Elementary abelian groups}

Galois ascent for stratification allows us to pass from tori to elementary abelian groups. This step is an elaboration of Mathew's approach~\cite[Sec.~6]{mathew_torus}.

\begin{Conv}
In this section, we will work over a fixed field $k$ and consider discrete commutative algebras $A$ over $k$. As before, we write $\bT$ for a torus of rank $r$ and $E \cong (\Z/p)^{\times r}$ as an elementary abelian $p$-group of rank $r$. If the characteristic of $k$ is different from $p$, the results of this subsection hold trivially, so we may tacitly assume $\mathrm{char}(k) = p$.
\end{Conv}

\begin{Lem}\label{lem:groupcohomologybasechange}
Let $A$ be a discrete commutative $k$-algebra. For any compact Lie group $G$, the canonical maps assemble into a pushout square
\[
\xymatrix{k \ar[r] \ar[d] & k^{hG} \ar[d] \\
A \ar[r] & A^{hG}}
\]
in the category of commutative $k$-algebras. Moreover, the induced square on homotopy groups is a pushout in the category of discrete commutative rings. 
\end{Lem}
\begin{proof}
Since $k$ is a field, there is a natural isomorphism $A \otimes_k \pi_*k^{hG} \simeq \pi_*(A \otimes_k k^{hG})$. Both parts of the statement are thus consequences of the observation that the canonical map
\[
\xymatrix{A\otimes_k H^{-*}(G;k) \cong A \otimes_k \pi_*k^{hG} \ar[r] & \pi_*A^{hG} \cong H^{-*}(G;A)}
\]
is an isomorphism. Indeed, this follows from the universal coefficient theorem for cohomology (see for example \cite[\S5.5, Thm.~11]{Spanier81}).
\end{proof}

\begin{Lem}\label{lem:elabgaloisext}
For any discrete commutative algebra $A \in \CAlg(k)$, the canonical map $A^{h\bT} \to A^{hE}$ is a faithful Galois extensions with Galois group $\bT$.
\end{Lem}
\begin{proof}
In \cite[Prop.~3.9]{mathew_torus} (see also \cite[\S9]{mathew_galois}), Mathew establishes this result for $A = k$. The general case follows from this by base-change. Indeed, there is a commutative diagram of commutative $k$-algebras
\[
\xymatrix{k \ar[r] \ar[d] & k^{h\bT} \ar[r] \ar[d] & k^{hE} \ar[d] \\
A \ar[r] & A^{h\bT} \ar[r] & A^{hE}}
\]
in which all maps are the canonical ones. By \cref{lem:groupcohomologybasechange} applied twice, the left inner square and outer rectangle are pushout squares in the category of commutative $k$-algebras. Therefore, the right inner square is a pushout as well. It follows from base-change for faithful Galois extensions, proven by Rognes in \cite[Lem.~7.1.1]{Rognes08}, that $A^{h\bT} \to A^{hE}$ is a faithful Galois extension with Galois group $\bT$, as desired.
\end{proof}

\begin{Thm}\label{thm:elabstratificationmodp}
Let $k$ be a field of characteristic $p$, let $A$ be a discrete Noetherian commutative algebra $A \in \CAlg(k)$, and $E \cong (\Z/p)^{\times r}$ an elementary abelian $p$-group of rank $r$. The tt-category $\Mod(A^{hE})$ is stratified over the spectrum
\[
\Spc(\Mod^c(A^{hE})) \cong \Spec^h(\pi_*A^{hE}).
\]
\end{Thm}
\begin{proof}
By \cref{lem:elabgaloisext}, the canonical map $A^{h\bT} \to A^{hE}$ is a faithful Galois extensions with Galois group $\bT$, so we can apply Galois ascent  (\cref{prop:galoisascent}). Indeed, $\Mod(A^{h\bT})$ is stratified as shown in \cref{prop:torusstratification}, hence so is $\Mod(A^{hE})$. In addition, the Balmer spectrum is then given by
\[
\Spc(\Mod^c(A^{hE})) \cong \Spc(\Mod^c(A^{h\bT})) \cong \Spec^h(\pi_*A^{h\bT}),
\]
see \eqref{eq:spctorus}. In fact, the map $\pi_*A^{h\bT} \to \pi_*A^{hE}$ induces a homeomorphism on homogeneous Zariski spectra, using the second part of \cref{lem:groupcohomologybasechange}. This implies the claim. 
\end{proof}


\begin{Rem}\label{rem:whygalois}
The case of $A=k$ of \cref{thm:elabstratificationmodp} is the technical heart of the paper \cite{BensonIyengarKrause11a} of Benson, Iyengar, and Krause. There, the authors use an elaboration of the Bernstein--Gelfand--Gelfand correspondence, modelled on a strategy by Avramov, Buchweitz, Iyengar, and Miller \cite{abim_bgg}, to prove their result. It is conceivable that this strategy extends to cover the case of (not necessarily regular) Noetherian $k$-algebras as well. 
\end{Rem}

\section{The passage through equivariant homotopy theory}\label{sec:eqperm}

One of the main observations that motivate this paper is that stratification for equivariant modules over equivariant ring spectra often reduces to non-equivariant stratification problems. Of particular relevance for us are the Borel-completions of commutative rings equipped with trivial $G$-action, which turn out to be amenable to this reduction technique. Combined with the fact that the categories $\Rep(G,R)$ arise via Borel completion from equivariant stable homotopy theory thus allows us to connect representation theory to non-equivariant stable homotopy theory. 

\begin{Conv}\label{conv:equivariant}
Let $\Sp_G$ be the symmetric monoidal stable $\infty$-category of genuine equivariant $G$-spectra for $G$ a finite group, see for example \cite[Part 2]{MathewNaumannNoel17} or the classical \cite{LewisMaySteinbergerMcClure86}. Denote the inflation functor by $\infl\colon \Sp \to \Sp_G$; this is the tt-functor which sends a spectrum $M$ to a $G$-spectrum with trivial $G$-action. If $R$ is a commutative ring spectrum, we write $\underline{R}_G = F(EG_+,\infl R) \in \Sp_G$ for the $G$-Borel equivariant commutative ring spectrum associated to $R$.
\end{Conv}

\subsection{Borel ring spectra}

Let $R$ be a Noetherian commutative ring. The category $\Mod_{\Sp_G}(\underline{R}_G)$ provides a gateway between the results of the previous section and the representation-theoretic categories considered below. It is an equivariant homotopy-theoretic incarnation of the category $D_{\mathrm{perm}}(G;R)$ of permutation modules studied by Balmer and Gallauer \cite{balmergallauer_permmodulesres,balmergallauer_permutationmodulescohomsing}. Indeed, first observe that $\Mod_{\Sp_G}(\underline{R}_G)$ is rigidly-compactly generated by the collection $\{\underline{R}_G \otimes G/H_+\}$ for $H$ running through the subgroups of $G$. Secondly, as proven in \cite[Cor.~6.21]{MathewNaumannNoel17}, there is a fully faithful tt-functor
\[
\xymatrix{\Mod_{\Sp_G}^c(\underline{R}_G) \ar[r] & \Fun(BG,\cat D(R)^c)}
\]
which sends $\underline{R}_G \otimes G/H_+$ to the permutation representation $R[G/H]$. This functor exhibits $\Mod_{\Sp_G}(\underline{R}_G)$ as the localizing subcategory of the derived $\infty$-category $\Ind\Fun(BG,\cat D(R)^c)$ generated by permutation modules. 

Let $k$ be a field of characteristic $p>0$ and consider a commutative algebra $A$ over $k$. Note that we do not impose any regularity assumptions on $A$.

\begin{Prop}\label{prop:monogenicborel}
Suppose $A$ is a discrete commutative algebra over a field $k$ of characteristic $p>0$ and let $P$ be a finite $p$-group. There is a canonical equivalence
\[
\xymatrix{\Mod_{\Sp_P}(\underline{A}_P) \ar[r]^-{\sim} & \Mod_{\Sp}(A^{hP})}
\]
of tt-categories. 
\end{Prop}
\begin{proof}
The result follows from derived Morita theory proven by Schwede and Shipley~\cite{SS03} and Lurie~\cite[Thm.~7.1.2.1]{HALurie}, as soon as we have verified the following two claims:
    \begin{enumerate}
        \item $\Mod_{\Sp_P}(\underline{A}_P)$ is compactly generated by its unit $\underline{A}_P$;
        \item the derived endomorphisms of $\underline{A}_P$ are equivalent to $A^{hP}$, as commutative ring spectra. 
    \end{enumerate}
For the first claim, note that the unit homomorphism $\eta\colon k \to A$ induces a tt-functor
\[
\xymatrix{\eta^*\colon \Mod_{\Sp_P}(\underline{k}_P) \ar[r] & \Mod_{\Sp_P}(\underline{A}_P),}
\]
whose forgetful right adjoint is conservative. In particular, $\eta^*$ preserves compact generators. Since $P$ is a finite $p$-group, $\Mod_{\Sp_P}(\underline{k}_P) \simeq \Mod_{\Sp}(k^{tP})$ is compactly generated by its unit. Therefore, the same is true for $\Mod_{\Sp_P}(\underline{A}_P)$.

In order to see (2), we compute
\[
\End_{\Mod_{\Sp_P}(\underline{A}_P)}(\underline{A}_P) \simeq \Hom_{\Sp_P}(S_P^0,\underline{A}_P) \simeq \Hom_{\Sp_P}(EP_+,\infl A) \simeq A^{hP}.
\]
As in \cite[\S2.2]{mathew_torus}, this equivalence is compatible with the commutative ring spectra structures on both sides, which implies the claim.
\end{proof}

\begin{Cor}\label{cor:monogenicgentate}
With notation as in \cref{prop:monogenicborel}, there is a symmetric monoidal equivalence
\[
\xymatrix{\Mod_{\Sp_P}(\underline{A}_P \otimes \widetilde{EP}) \ar[r]^-{\sim} & \Mod_{\Sp}(A^{tP}),}
\]
where $\widetilde{EP}$ denotes the cofiber of the canonical map $EP_+ \to S^0$.
\end{Cor}
\begin{proof}
Under the equivalence of \cref{prop:monogenicborel}, the object $\underline{A}_P \otimes P_+$ maps to the $A^{hP}$-module spectrum $(\underline{A}_P\otimes P_+)^{hP} \simeq A$, with action induced by the augmentation. The corresponding finite localizations away from the localizing ideals generated by $\underline{A}_P \otimes P_+$ and $A$, respectively, fit into a commutative square
\[
\xymatrix{\Mod_{\Sp_P}(\underline{A}_P) \ar[r]^-{\sim} \ar[d]_{\theta^*} & \Mod_{\Sp}(A^{hP}) \ar[d]^{\theta^*} \\
\Mod_{\Sp_P}(\underline{A}_P \otimes \widetilde{EP}) \ar@{-->}[r]_-{\sim} & \Mod_{\Sp}(A^{tP}).}
\]
The induced bottom horizontal equivalence is the desired equivalence.
\end{proof}

\begin{Rem}
Alternatively, one can also run the proof \cref{prop:monogenicborel} directly for $\Mod_{\Sp_P}(\underline{A}_P \otimes \widetilde{EP})$, using that the derived endomorphisms are given by
\[
\End_{\Mod_{\Sp_P}(\underline{A}_P \otimes \widetilde{EP})}(\underline{A}_P \otimes \widetilde{EP}) \simeq A^{tP},
\]
as commutative ring spectra.
\end{Rem}

\subsection{Stratification for regular coefficients}

The goal of this subsection is to establish stratification for the tt-categories $\Mod_{\Sp_{E}}(\underline{R}_E)$ for elementary abelian $p$-groups $E$. The proof naturally splits into two cases: either $p$ is $0$ in $R$ or it is not. We will consider these two cases separately, beginning with the case that $p$ vanishes in $R$, where in fact we can say something more general. The other case then follows from the modular one through a lifting argument. 

\subsubsection*{Characteristic $p>0$}

\begin{Prop}\label{prop:elabeqstratificationcharp}
Let $k$ be a field of characteristic $p$, let $A$ be a discrete Noetherian commutative algebra $A \in \CAlg(k)$. If $E$ is an elementary abelian $p$-group, then the tt-category $\Mod_{\Sp_{E}}(\underline{A}_E)$ is stratified over its spectrum
\[
\Spc(\Mod_{\Sp_{E}}^c(\underline{A}_E)) \cong \Spec^h(\pi_*A^{hE}).
\]
\end{Prop}
\begin{proof}
\Cref{prop:monogenicborel} provides a symmetric monoidal equivalence 
\[
\xymatrix{\Mod_{\Sp_E}(\underline{A}_E) \ar[r]^-{\sim} & \Mod_{\Sp}(A^{hE})},
\]
so the claim follows from \cref{thm:elabstratificationmodp}.
\end{proof}

\begin{Rem}
Observe that we do not impose any regularity assumptions on $A$ in this proposition. In particular, for elementary abelian $p$-groups and characteristic $p$ coefficients $A$, we obtain a partial generalization of Lau's computation \cite{lau_spcdmstacks} of the Balmer spectrum of $\Spc(\Mod_{\Sp_{E}}^c(\underline{A}_E))$ for regular $A$.
\end{Rem}

\subsubsection*{The general case}

\begin{Lem}\label{lem:injectivity}
Let $R$ be a regular commutative ring, $p$ a prime number, and $E$ an elementary abelian $p$-group of rank $r \ge 0$. The reduction map $q\colon R \to R/p$ induces an injective map 
\[
\xymatrix{\Spc(q^*)\colon \Spc(\Mod_{\Sp_{E}}^c(\underline{R/p}_E)) \ar[r] & \Spc(\Mod_{\Sp_{E}}^c(\underline{R}_E))}
\]
on Balmer spectra.
\end{Lem}
\begin{proof}
Let $\bT = (S^{1})^{\times r}$ be a torus of rank $r$ and view $E \subset \bT$ via the embedding of the $p$-th roots of unity. By naturality, the reduction map $q\colon R\to R/p$ fits into a commutative square of graded commutative cohomology rings:
\begin{equation}\label{eq:lem:injectivity1}
\vcenter{
\xymatrix{H^*(E;R) \ar[d] & H^*(\bT;R) \ar[l] \ar[d] \\
H^*(E;R/p) & H^*(\bT;R/p). \ar[l] }
}
\end{equation}
First, we claim that the bottom horizontal map is an $F$-isomorphism. Indeed, the proof of \cref{lem:groupcohomologybasechange} shows that $H^*(\bT;R/p) \to H^*(E;R/p)$ identifies with
\[
H^*(\bT;\F_p) \otimes_{\F_p} R/p   \to H^*(E;\F_p) \otimes_{\F_p} R/p,
\]
base-changed from the $F$-isomorphism $H^*(\bT;\F_p) \to H^*(E;\F_p)$. This implies the claim. Second, the right vertical map in \eqref{eq:lem:injectivity1} may be identified with the projection
\[
R[y_1,\ldots,y_r] \to R/p[y_1,\ldots,y_r]
\]
with $|y_i| = -2$ for all $i$. Applying homogeneous Zariski spectra to the square \eqref{eq:lem:injectivity1} thus yields a commutative diagam
\begin{equation}\label{eq:lem:injectivity2}
    \xymatrix{\Spec^h(H^*(E;R)) \ar[r] & \Spec^h(H^*(\bT;R)) \\
    \Spec^h(H^*(E;R/p)) \ar[r]_-{\cong} \ar[u]^-{\Spec^h(H^*(q))} & \Spec^h(H^*(\bT;R/p)), \ar@{^{(}->}[u]}
\end{equation}
in which the composite factoring through the right bottom term is injective. It follows that $\Spec^h(H^*(q))$ is injective as well.\footnote{Alternatively, this statement can also be verified by direct computation.}

Let $q^*\colon \Mod_{\Sp_{E}}^c(\underline{R}_E) \to \Mod_{\Sp_{E}}^c(\underline{R/p}_E)$ denote base-change along $q$. Naturality of Balmer's comparison map \eqref{eq:Balmercomparisonmap} then provides a commutative square of spectra
\[
\xymatrixcolsep{5pc}{
\xymatrix{\Spc(\Mod_{\Sp_{E}}^c(\underline{R/p}_E)) \ar[r]^-{\Spc(q^*)} \ar[d]_{\rho^{\bullet}}^{\cong} & \Spc(\Mod_{\Sp_{E}}^c(\underline{R}_E)) \ar[d]_{\cong}^{\rho^{\bullet}} \\
\Spec^h(H^*(E;R/p)) \ar@{^{(}->}[r]_{\Spec^h(H^*(q))} & \Spec^h(H^*(E;R)).}}
\]
Note that the vertical comparison maps are homeomorphisms by Lau's theorem and \cref{thm:elabstratificationmodp}, respectively. It follows that $\Spc(q^*)$ is injective, as desired.
\end{proof}

\begin{Thm}\label{thm:elabeqstratification}
Let $R$ be a regular commutative ring. If $E$ is an elementary abelian $p$-group, then the tt-category $\Mod_{\Sp_{E}}(\underline{R}_E)$ is stratified over its spectrum
\[
\Spc(\Mod_{\Sp_{E}}^c(\underline{R}_E)) \cong \Spec^h(\pi_*R^{hE}).
\]
\end{Thm}
\begin{proof}
Any regular commutative ring decomposes as a finite product of regular integral domains. Since the formation of the categories $\Mod_{\Sp_G}(\underline{R}_G)$ commutes with finite products, we can reduce to $R$ itself being a regular domain. If $p$ is zero in $R$, then the result is a special case of \cref{prop:elabeqstratificationcharp}, so it remains to consider the case that $p \in R$ is regular.

In order to prepare for the proof, we begin with the construction of two tt-functors, $q^*$ and $f^*$, with the following two properties:
    \begin{enumerate}
        \item on Balmer spectra, the functors $q^*$ and $f^*$ induce a jointly bijective map;
        \item the target categories of $q^*$ and $f^*$ are stratified. 
    \end{enumerate}
We will then argue by descent, appealing to the techniques from \cref{ssec:ttascentanddescent}.    

The quotient map $R \to R/p$ induces a morphism $q\colon \underline{R}_E \to \underline{R/p}_E$ of commutative $E$-equivariant ring spectra, and hence a colimit preserving tt-functor
\[
\xymatrix{q^*\colon\Mod_{\Sp_{E}}(\underline{R}_E) \ar[r] & \Mod_{\Sp_{E}}(\underline{R/p}_E).}
\]
Its forgetful right adjoint $q_*$ is conservative and preserves compact objects: Indeed, $\Mod_{\Sp_{E}}(\underline{R/p}_E)$ is monogenic by \cref{prop:monogenicborel} and $\underline{R/p}_E \simeq \cof(\underline{R}_E \xrightarrow{p} \underline{R}_E)$, viewed as an $\underline{R}_E$-module, is compact. This suffices to apply \cite[Thm.~1.7]{Balmer18}, which shows that 
\begin{equation}\label{eq:suppofRmodp}
    \supp(\underline{R/p}_E) = \im(\varphi);
\end{equation}
here, $\varphi\colon \Spc(\Mod_{\Sp_{E}}^c(\underline{R/p}_E)) \to \Spc(\Mod_{\Sp_{E}}^c(\underline{R}_E))$ is the map on spectra induced by $q^*$. By \cref{lem:injectivity}, $\varphi = \Spc(q^*)$ is injective.

Consider the finite (Verdier) localization of $\Mod_{\Sp_{E}}(\underline{R}_E)$ away from the localizing ideal $\Loco{\underline{R/p}_E}$:
\[
\xymatrix{f^*\colon \Mod_{\Sp_{E}}(\underline{R}_E) \ar[r] & \Mod_{\Sp_{E}}(\underline{R[1/p]}_E).}
\]
The category of local objects is as stated, because localizing away from the cofiber of multiplication by $p$ has the effect of inverting $p$, which in turn can be performed on the level of the unit. By virtue of \cref{prop:monogenicborel} and the fact that the group cohomology ring $\pi_*R^{hE} \cong H^*(E;R)$ is $p^{\rank_p(E)}$-torsion in all non-zero degrees, we can identify the target of $f^*$ further as 
\[
\Mod_{\Sp_{E}}(\underline{R[1/p]}_E) \simeq \Mod_{\Sp}((R[1/p])^{hE}) \simeq \Mod_{\Sp}(R[1/p]).
\]
The induced injective map $\psi\colon \Spc(\Mod_{\Sp}^c(R[1/p])) \to \Spc(\Mod_{\Sp_{E}}^c(\underline{R}_E))$ on spectra exhibits the domain as the quasi-compact open subset complementing $\supp(\underline{R/p}_E)$. Combined with \eqref{eq:suppofRmodp} this shows that $\varphi$ and $\psi$ are jointly surjective, thereby establishing Property (a) above. Property (b) holds as well, using \cref{prop:elabeqstratificationcharp} and Neeman's stratification theorem for derived categories of Noetherian commutative rings, \cite[Thm.~2.8]{Neeman92a}.

After this preparation, we can now descend stratification along $f^*$ and $q^*$. To this end, first consider a local category $\Gamma_{\frak p}\Mod_{\Sp_{E}}(\underline{R}_E)$ for a prime tt-ideal $\frak p \in \im(\psi) \subseteq \Spc(\Mod_{\Sp_{E}}^c(\underline{R}_E))$. Since $\Mod_{\Sp}(R[1/p])$ is stratified, $\Gamma_{\frak p}\Mod_{\Sp_{E}}(\underline{R}_E)$ is minimal by Zariski descent \cref{prop:ttdescent}(1). If $\frak p \in \im(\varphi)$ instead, we deduce minimality of $\Gamma_{\frak p}\Mod_{\Sp_{E}}(\underline{R}_E)$ via \'etale descent along $f^*$, using \cref{prop:ttdescent}(2). 

Finally, it remains to observe that $\Spc(\Mod_{\Sp_{E}}^c(\underline{R}_E))$ is Noetherian and homeo-morphic to $\Spec^h(\pi_*R^{hE})$, thanks to Lau's theorem \cite{lau_spcdmstacks}.
\end{proof}

\section{Representation categories with regular coefficients}\label{sec:reptheory}

We are finally ready to deduce stratification for the categories $\Rep(G,R)$ of representations for regular coefficients $R$. In addition to \cref{thm:elabeqstratification}, the key inputs are a `generation by permutation modules' result to pass from equivariant modules over the Borel-complete ring spectrum $\underline{R}_G$ to representation theory, as well as finite \'etale descent from \cite{barthel_integral1}.

\subsection{Representation-theoretic interpretation}\label{ssec:derreptheory}

Write $\Perf(R) = \Mod^c(R)$ for the symmetric monoidal category of perfect $R$-modules. The \emph{derived $\infty$-category of $R$-linear $G$-representations} is defined as
\[
\Rep(G,R) = \Ind\Fun(BG,\Perf(R)),
\]
where $\Ind$ denotes ind-completion~\cite[\S5.3.5]{HTTLurie}. This category has the structure of a rigidly-compactly generated symmetric monoidal stable $\infty$-category. 

\begin{Rem}\label{rem:incarnations}
We offer some comments on the derived category of representations considered in this paper. Since it is rigidly-compactly generated, $\Rep(G,R)$ is determined as a symmetric monoidal category by its full subcategory of compact objects $\rep(G,R) = \Fun(BG,\Perf(R))$. Its homotopy category $\mathrm{Ho}(\rep(G,R))$ can be identified with the following categories:
    \begin{enumerate}
        \item The bounded derived category of modules over $R[G]$ whose underlying complex of $R$-modules is perfect. 
        \item The category of perfect complexes over the Deligne--Mumford quotient stack $[\Spec(R)/G]$, when $R$ is a regular commutative ring equipped with trivial $G$-action.
    \end{enumerate}
Furthermore, $\Rep(G,R)$ admits a canonical localization functor to the unbounded derived category $D(R[G])$ whose kernel is by definition our version of the stable module category $\StMod(G,R)$ with coefficients in a general commutative ring $R$. In fact, this localization sequence promotes to a recollement familiar from Rickard's presentation of the stable module category with field coefficients. 

The stable module category $\StMod(G,R)$ also admits an algebraic model; we refer to \cite[Sec.~3.3]{barthel_integral1} for the details. In short, let $\Lat(G,R)$ be the category of $R$-linear $G$-representations with underlying projective $R$-module. The category of such representations inherits a Frobenius exact structure from the abelian category $\Mod(R[G])$ in which it is embedded. Likewise, the $R$-linear tensor product on $\Mod(R[G])$ restricts to a symmetric monoidal structure on $\Lat(G,R)$. The homotopy category of the associated Quillen exact model structure is then canonically equivalent to the homotopy category of the $\infty$-category $\StMod(G,R)$ as tt-categories.
\end{Rem}

The next result summarizes the discussion of Sections 3.1 and 3.2 in \cite{barthel_integral1}, specifically Theorem 3.7, Proposition 3.16, and Corollary 3.21 therein.

\begin{Prop}[\cite{barthel_integral1}]\label{prop:eqreptranslation}
Let $G$ be a finite group and $R$ a regular commutative ring. There is a commutative diagram of colimit preserving tt-functors
\begin{equation}\label{eq:equivtorep}
    \vcenter{
    \xymatrix{\Mod_{\Sp_G}(\underline{R}_G) \ar[r]^-{\sim} \ar[d]_{\theta^*} & \Rep(G,R) \ar[d]^{\theta^*} \\
    \Mod_{\Sp_G}(\underline{R}_G \otimes \widetilde{EG}) \ar[r]_-{\sim} & \StMod(G,R)}
    }
\end{equation}
in which the vertical maps are finite localizations away from $\Loco{\underline{R}_G\otimes G_+}$ and $\Loco{R[G]}$, respectively. 
\end{Prop}

\begin{Rem}
The top equivalence in Diagram \eqref{eq:equivtorep} is a homotopical manifestation of the generation by permutation modules principle proven by Rouquier~\cite{Rouquier_letter}, Mathew~\cite[App.~A]{treumannmathew_reps}, and Balmer--Gallauer~\cite{balmergallauer_permmodulesres}. More generally, for coefficients in any commutative ring spectrum $S$, there is a colimit preserving tt-functor
\[
\xymatrix{\Mod_{\Sp_G}(\underline{S}_G) \ar[r] & \Rep(G,S)}
\]
which is fully faithful. The question of when this functor is also essentially surjective will be studied further in joint work in progress with Gallauer. 
\end{Rem}

This proposition allows us to assemble the desired stratification results from the theorems proven in \cref{sec:eqperm} coupled with Lau's computation of the spectrum of $\rep(G,R) = \Rep(G,R)^c$.

\begin{Thm}\label{thm:repstratification}
Suppose $G$ is a finite group and $R$ is a regular commutative ring, then $\Rep(G,R)$ is stratified over $\Spc(\rep(G,R)) \cong \Spec^h(H^*(G;R))$.
\end{Thm}
\begin{proof}
If $G = E$ is an elementary abelian $p$-group for some prime $p$, the category $\Mod_{\Sp_E}(\underline{R}_E)$ is stratified by \cref{thm:elabeqstratification}. Combined with \cref{prop:eqreptranslation}, this implies the result for $\Rep(E,R)$. The case of a general finite group then follows from this via finite \'etale descent (\cref{prop:ttdescent}(3)), as shown in \cite[Thm.~4.18(1)]{barthel_integral1}. Finally, the identification of the spectrum $\Spc(\rep(G,R))$ is the content of Lau's theorem \cite{lau_spcdmstacks}.
\end{proof}

\subsection{Classification and further consequences}\label{ssec:applications}

This short subsection collects consequences of our main theorem for $\Rep(G,R)$ and related categories.

\begin{Thm}\label{thm:repconsequences}
Let $G$ be a finite group and let $R$ be a regular commutative ring. The universal support function composed with Balmer's comparison map induces a bijection
\[
\Supp\colon\begin{Bmatrix}
\text{Localizing tensor ideals} \\
\text{of } \cat \Rep(G,R)
\end{Bmatrix} 
\xymatrix@C=2pc{ \ar[r]^-{\sim} &}
\begin{Bmatrix}
\text{Subsets of}  \\
\Spec^h(H^*(G;R))
\end{Bmatrix}.
\]
Moreover, the telecope conjecture holds in $\Rep(G,R)$, so that we obtain bijections
\[
\begin{Bmatrix}
\text{Smashing ideals} \\
\text{of } \cat \Rep(G,R)
\end{Bmatrix} 
\simeq
\begin{Bmatrix}
\text{Thick tensor ideals} \\
\text{of } \cat \rep(G,R)
\end{Bmatrix} 
\simeq
\begin{Bmatrix}
\text{Specialization closed}  \\
\text{subsets of } \Spec^h(H^*(G;R))
\end{Bmatrix},
\]
and support satisfies the tensor product formula for any $M,N \in \Rep(G,R)$:
\[
\Supp(M \otimes N) = \Supp(M) \cap \Supp(N).
\]
\end{Thm}
\begin{proof}
In light of \cref{thm:repstratification}, the first statement holds by definition of stratification, while the rest follow formally from \cref{thm:repstratification} by virtue of \cite[Thm.~8.1]{bhs1} and \cite[Thm.~9.11]{bhs1}.
\end{proof}

We obtain a similar classification result for the integral stable module category of a finite group with regular coefficients. 

\begin{Thm}\label{thm:stmodstratification}
Let $R$ be a regular ring and $G$ a finite group. The stable module category $\StMod(G,R)$ is stratified over
\[
\Spc(\StMod^c(G,R)) \cong \Proj(H^*(G;R)).
\]
In particular, the telescope conjecture and the tensor product formula for cohomological support hold in $\StMod(G,R)$.  
\end{Thm}
\begin{proof}
Since stratification ascends along finite localizations (\cite[Cor.~5.5]{bhs1}), we see that $\StMod(G,R)$ is stratified by \cref{prop:eqreptranslation} and \cref{thm:repstratification}. Moreover, \cite[Cor.~3.32]{barthel_integral1} provides a natural homeomorphism
\[
\Spc(\StMod^c(G,R)) \cong \Proj(H^*(G;R)),
\]
thus finishing the proof.
\end{proof}

Finally, we can give the desired classification for the category $\Lat(G,R)$ of ordinary $R$-linear $G$-representations with underlying projective $R$-module, whose construction was recalled in \cref{rem:incarnations}. 

\begin{Cor}\label{cor:classicalrepconsequences}
Let $G$ be a finite group and $R$ a regular commutative ring. Cohomological support induces  bijections 
\[
\begin{Bmatrix}
\text{Non-zero localizing tensor } \\
\text{ideals of } \cat \Lat(G,R)
\end{Bmatrix} 
\xymatrix@C=2pc{ \ar[r]^-{\sim} &}
\begin{Bmatrix}
\text{Subsets of}  \\
\Proj(H^*(G;R))
\end{Bmatrix}
\]
and
\[
\begin{Bmatrix}
\text{Non-zero thick} \\
\text{ideals of } \cat \lat(G,R)
\end{Bmatrix} 
\xymatrix@C=2pc{ \ar[r]^-{\sim} &}
\begin{Bmatrix}
\text{Specialization closed}  \\
\text{subsets of } \Proj(H^*(G;R))
\end{Bmatrix}.
\]
\end{Cor}
\begin{proof}
In light of \cref{thm:stmodstratification}, the desired bijections are furnished by \cite[Lem.~6.23 and Rem.~6.25]{barthel_integral1}.
\end{proof}

\bibliographystyle{alpha}\bibliography{TG-articles}

\end{document}